\documentclass[12pt]{amsart}

\usepackage{amsmath}
\usepackage{amsthm}
\usepackage{hyperref}
\usepackage{amssymb}
\usepackage[margin=1.0in]{geometry}
\usepackage{tikz}

\numberwithin{equation}{section}
\newtheorem{prop}{Proposition}
\newtheorem{lemma}[prop]{Lemma}

\newtheorem{thm}[prop]{Theorem}
\newtheorem{cor}[prop]{Corollary}
\newtheorem{conj}[prop]{Conjecture}
\numberwithin{prop}{section}

\theoremstyle{definition}
\newtheorem{defn}[prop]{Definition}

\newtheorem{rmk}[prop]{Remark}

\newcommand{\dt}{\frac{\partial}{\partial t}}
\newcommand{\brs}[1]{\left| #1 \right|}

\newcommand{\gD}{\Delta}

\newcommand{\gs}{\sigma}

\newcommand{\gw}{\omega}

\newcommand{\til}[1]{\widetilde{#1}}

\renewcommand{\bar}[1]{\overline{#1}}

\DeclareMathOperator{\tr}{tr}

\begin{document}

\title[Pluriclosed flow on Oeljeklaus-Toma manifolds]{Pluriclosed flow on Oeljeklaus-Toma manifolds}

\author{Jeffrey Streets}
\address{Rowland Hall\\
         University of California, Irvine\\
         Irvine, CA 92617}
\email{\href{mailto:jstreets@uci.edu}{jstreets@uci.edu}}

\author{Xiaokang Wang}
\address{Rowland Hall\\
         University of California, Irvine\\
         Irvine, CA 92617}
\email{\href{mailto:xiaokaw@uci.edu}{xiaokaw@uci.edu}}

\date{\today}

\begin{abstract} We establish global existence of the pluriclosed flow with arbitrary initial data on Oeljeklaus-Toma manifolds, and Gromov-Hausdorff convergence of blowdown limits to a torus under natural conjectural bounds on the flow at infinity.  In the case of generalized K\"ahler-Ricci flow we prove refined a priori estimates in support of these conjectural bounds.
\end{abstract}

\thanks{JS is supported by the NSF via DMS-2342135.  The authors warmly dedicate this article to Professor Gang Tian on the occasion of his 65th birthday.}

\maketitle

\section{Introduction}

In recent years the pluriclosed flow \cite{PCF, PCFReg} and generalized K\"ahler-Ricci flow \cite{apostolov2022generalized,StreetsSTB,GKRF} have been developed as a tool for understanding the geometry of complex, especially non-K\"ahler, manifolds \cite{barbaro2023bismut,barbaro2025global,barbaro2025pluriclosed,fino2024pluriclosed,fusi2024pluriclosed,garcia2023non, ye2024pluriclosed}.  A natural class of non-K\"ahler manifolds are the Oeljeklaus-Toma (OT) manifolds \cite{oeljeklaus2005non}, whose geometry is linked to the structure of number fields, and which are natural higher dimensional generalizations of Inoue surfaces \cite{inoue1974surfaces}.  In \cite{fusi2024pluriclosed} a complete description of the pluriclosed flow with left-invariant initial data on OT manifolds was obtained, in particular showing that the solution exists for all time and collapses after blowdown to a torus in the Gromov-Hausdorff sense.  Moreover the blowdown on the universal cover converges in the Cheeger-Gromov sense to a soliton.  It is natural to conjecture that these statements hold for arbitrary initial data (cf. Conjecture \ref{c:OT_PCF}).  In this work we confirm some aspects of this conjecture.  

The first main result is to establish the global existence of the flow:

\begin{thm} \label{t:mainthm1} Fix $(M^{2n}, J)$ an OT manifold and $g_0$ a pluriclosed metric on $M$.  The solution to pluriclosed flow with initial data $g_0$ exists on $[0, \infty)$.
\end{thm}

\noindent The proof exploits natural class of background metrics arising from the homogeneous structure on OT-manifolds.  In the case of Inoue surfaces these are known as Tricerri metrics \cite{tricerri1982some}.  In the next theorem we give some refined estimates in the special case of generalized K\"ahler-Ricci flow (GKRF) \cite{GKRF}.  Oeljeklaus-Toma manifolds admit natural classes of generalized K\"ahler structures with split tangent bundle, and for such metrics the GKRF reduces to a scalar parabolic flow of Monge-Amp\`ere type \cite{StreetsSTB}.

\begin{thm} \label{t:mainthm2} For generalized K\"ahler-Ricci flow on an Oeljeklaus-Toma manifold
\begin{enumerate}
\item The scalar potential $\phi$ satisfies
\begin{align*}
    - C \leq \phi \leq C e^{-t}(1 + t).
\end{align*}
\item Assuming there exists a Tricerri-type metric in $[\gw_0]$, we have
\begin{align*}
    - C e^{-t}(1 + t) \leq \phi \leq C e^{-t}(1 + t).
\end{align*}
\item On Inoue surfaces of type $S_M$, the estimate of item (2) holds. In addition:
\begin{align*}
    -C\leq \dot{\phi}\leq C.
\end{align*}
\item On Inoue surfaces of type $S_M$, we have:
\begin{align*}
        \omega(t)\geq C\omega_h(t),
\end{align*}
where $\omega_h(t)$ is the model flow with initial data the Tricerri metric $h$.
\end{enumerate}
\end{thm}

\section{Geometry of Oeljeklaus-Toma manifolds} \label{s:OT}
\subsection{Definition}

In this section, we recall the family of compact non-Kähler complex manifolds constructed by Oeljeklaus-Toma \cite{oeljeklaus2005non}.  First, we need some facts from algebraic number theory. Let $K$ be an algebraic number field, i.e. $K\simeq \mathbb{Q}[x]/(f)$, where $f\in \mathbb{Q}[x]$ is a monic irreducible polynomial of degree $s+2t=[K:\mathbb{Q}]$, where
    \begin{align*}
        s=\#\text{ real roots },\quad 2t=\#\text{ complex roots},
    \end{align*}
    which for arbitrary given $s$ and $t$ exists by (\cite{oeljeklaus2005non} Remark 1.1).

    Now, consider the embedding, $K \hookrightarrow \mathbb{Q}$, given by the roots of $f$: Let $a_1,...,a_s\in \mathbb{R}$ be the real roots of $f$, and let $a_{s+t+1}=\bar{a_{s+1}}, ..., a_{s+2t}=\bar{a_{s+t}}\in\mathbb{C}$  be the complex roots of $f$. Define the embedding:
    \begin{align*}
        \sigma_i: K\to \mathbb{C}, \quad x\mapsto a_i.
    \end{align*}
    Thus, $\sigma_1,\dots,\sigma_s : K\to \mathbb{R}$, and $\sigma_{s+1},...,\sigma_{s+2t} : K \to \mathbb C,$ where $\sigma_{s + i} = \bar{\sigma_{s + t + i}}$. There exist $n:=s+2t$ different embeddings.

    Let $\mathcal{O}_K\subset K$ be the ring of algebraic integers. Note $\mathcal{O}_K$ is a finitely generated free abelian group of rank $n=s+2t$, i.e. $\mathcal{O}_k\simeq \mathbb{Z}^{n}$. (The rank of $\mathcal{O}_k$ as a free $\mathbb{Z}$-module is $[K:\mathbb{Q}]=s+2t$). Moreover, let $\mathcal{O}_K^*\subseteq \mathcal{O}_K$ be the multiplicative group of units of $\mathcal{O}_K$.  By the Dirichlet unit theorem, for $s \geq 1$ we have $\mathcal O_K^* \cong \{\pm 1\} \times \mathbb Z^{t + s - 1}$.  We define
    \begin{align*}
        \mathcal{O}_K^{*,+}:=\{a \in \mathcal{O}_K^*\ |\ \sigma_i(a)>0,\ 1\leq i\leq s\}.
    \end{align*}
    By definition, $\mathcal{O}_K^{*,+}$ is a finite index subgroup of $\mathcal{O}_K^*$.

    Now, let $\mathbb{H}=\{z\in\mathbb{C}:\operatorname{Im} z >0\}\}$.  We define two kinds of actions on $\mathbb{H}^s\times \mathbb{C}^t$:
    \begin{align*}
        &T: \mathcal{O}_K \to \operatorname{Aut}(\mathbb{H}^s\times \mathbb{C}^t)\\
        &T(a)=[(w_1,...,w_s,z_{s+1},...,z_{s+t})\mapsto (w_1+\sigma_1(a),...,z_{s+t}+\sigma_{s+t}(a))],
    \end{align*}
    and
    \begin{align*}
        &R:\mathcal{O}_K^{*,+}\to \operatorname{Aut}(\mathbb{H}^s\times \mathbb{C}^t)\\
        &R(u)=[(w_1, ..., w_s, z_{s+1}, ..., z_{s+t})\mapsto (w_1\cdot\sigma_1(u),...,z_{s+t}\cdot\sigma_{s+t}(u))].
    \end{align*}
    Let
    \begin{align*}
        \sigma : K \to \mathbb{C}^{s+t}, \quad 
        & a\mapsto (\sigma_1(a), ..., \sigma_{s+t}(a)).
    \end{align*}
    Note $\sigma(\mathcal{O}_K)$ is a lattice of rank $s+2t$ in $\mathbb{C}^{s+t}$. Thus, we have $\operatorname{rank}(T(\mathcal{O}_K))=s+2t$. Note $R(u)\sigma(\mathcal{O}_K) = \sigma(\mathcal{O}_K)$, by the property of the algebraic integer ring structure. Then:
    \begin{align*}
        \mathcal{O}_K^{*,+}\ltimes \mathcal{O}_K
    \end{align*}
    acts on $\mathbb{H}^s\times \mathbb{C}^t$.  By \cite{oeljeklaus2005non} there exists a subgroup $U\leq \mathcal{O}_K^{*,+}$, $\operatorname{rank}(U)=s$, such that  $U\ltimes \mathcal{O}_K$ is cocompact and acts properly discontinuously on $\mathbb{H}^s\times \mathbb{C}^t$.  The compact quotient:
    \begin{align*}
        X(K,U)=\mathbb{H}^s\times \mathbb{C}^t/U\ltimes \mathcal{O}_K
    \end{align*}
    is called an \textit{Oeljeklaus-Toma manifold} of type $(s,t)$. Note:
    \begin{align*}
        \mathbb{H}^s\times\mathbb{C}^t/\sigma(\mathcal{O}_K)    \cong (\mathbb{R}_{>0})^s\times(S^1)^{ 2t+s}
    \end{align*}
    and $U$ acts on $(\mathbb{R}_{> 0})^s$ preserves the fiber by the property of the algebraic integer structure. Thus, we may regard $X(K,U)$ as a torus $\mathbb{T}^t$ bundle over $\mathbb{T}^s$.

    When $s=t=1$, it is the famous \textit{Inoue surface} of type $S_M$ \cite{inoue1974surfaces}. The original $S_M$ is constructed in a more direct way:
    consider $M\in \operatorname{SL}(3,\mathbb{Z})$, $M=(m_{ik})$. Denote the eigenvalues: $\lambda, \mu, \Bar{\mu}$ with the relation: $\lambda\cdot |\mu|^2=1$. In addition, we require $\mu\neq \Bar{\mu}$ and $\lambda >1$. The corresponding unit eigenvectors: $(a_1,a_2,a_3)$, $(b_1, b_2, b_3)$ and $(\Bar{b_1}, \Bar{b_2}, \Bar{b_3})$. Define the actions:
    \begin{align*}
        &g_0:(z,w)\mapsto (\mu z, \lambda w)\\
        &g_i:(z,w)\mapsto (z+b_i, w+a_i)
    \end{align*}
    for $i=1,2,3$. Note that $\{\left(\begin{aligned}
        a_i\\
        b_i
    \end{aligned}\right)\}$ are linearly independent and satisfy:
    \begin{align*}
        \left(\begin{aligned}
            \lambda a_i\\
            \mu b_i
    \end{aligned}
    \right) = \sum_{k=1}^3 m_{ik}\left(\begin{aligned}
        a_k\\
        b_k
    \end{aligned}
    \right).
    \end{align*}
    Then, the quotient $M:=\mathbb{C}\times\mathbb{H}/\langle g_0, g_1, g_2, g_3\rangle$ is a compact complex surface, which is the desired Inoue surface $S_M$. The relation between two constructions can be seen in this way:
    choose $f$ as the characteristic polynomial of $SL(3,\mathbb{Z})$. Note that when $s=t=1$, then $\mathcal{O}_K^{*,+}\simeq\mathbb{Z}$. Choose a generator $u\in\mathcal{O}_K^{*,+}$, by Dirichlet unit theorem, $\sigma_1(u)\cdot|\sigma_2(u)|^2=1$. Then, any choice of the basis of $\mathcal{O}_K\simeq\mathbb{Z}^3\subset\mathbb{R}\times\mathbb{C}$ will correspond to an element of $SL(3,\mathbb{Z})$.

    \subsection{Pluriclosed and generalized K\"ahler structures on Oeljeklaus-Toma manifolds} \label{s:OT_PC_metrics}

    Consider the pluriclosed metrics on OT manifolds.  Note that not all OT manifolds admit pluriclosed metrics.  In fact:
    \begin{thm}[\cite{otiman2022special, angella2024metric}]
    Let $X(K,U)$ be an OT manifold of type $(s,t)$. It admits pluriclosed metrics if and only if $s=t$ and, for any $u\in U$,
    \begin{align*}
        \sigma_j(u)|\sigma_{s+j}(u)|^2=1, \quad \text{for any }i\in\{1,...,s\}.
    \end{align*}
\end{thm}
\noindent Note that for Inoue surfaces $S_M$, this is automatically true. From now on, we restrict ourselves to the case where $s=t$, i.e. the universal cover of the OT manifolds is $\mathbb{H}^s\times\mathbb{C}^s$.

Some pluriclosed metrics can be constructed explicitly on OT manifolds of type $(s,s)$:
\begin{align}\label{e:model metric}
    \omega_h = \sum_{i=1}^s\sqrt{-1}\frac{dw_i\wedge d\bar{w_i}}{(\operatorname{Im}w_i)^2} + \sqrt{-1}\operatorname{Im}w_i dz_i\wedge d\bar{z_i}.
\end{align}
It is easy to verify $\partial\bar{\partial}\omega_h=0$, hence it is pluriclosed.  Moreover, this metric is invariant under the group action $U\ltimes \mathcal{O}_K$, and therefore descends to the quotient. When $s=1$, the metric $\omega_h$ is the pluriclosed metric discovered by Tricerri \cite{tricerri1982some} on the Inoue Surface. 

More generally, for any given sequence $a = \{a_1,...,a_s\}$, and $b=\{b_1,...,b_s\}$, where each $a_i$, $b_i\in \mathbb{R}^+$. We can define the following metric:
\begin{align}\label{e:AB model initial metric}
    \omega^{a,b}_h = \sum_{i=1}^s\sqrt{-1}a_i\frac{dw_i\wedge d\bar{w_i}}{(\operatorname{Im}w_i)^2} + \sqrt{-1}b_i\operatorname{Im}w_i dz_i\wedge d\bar{z_i}.
\end{align}
Note, the family of metrics $\omega^{a,b}_h$ are pluriclosed as well. Thus, we have a family of pluriclosed metrics on $X(K,U)$. We shall write $\omega_h$ for short if $a_i$, $b_i = 1$.  We record some curvature facts of these model metrics, the proof of which is a standard computation.

\begin{prop} \label{p:modelcurvature} The metric $\gw^{a,b}_h$ satisfies:
\begin{enumerate}
    \item The only nonvanishing components of the Chern curvature are
    \begin{align*}
        &\Omega^{Ch}_{w_i\bar{w_i}w_i\bar{w_i}}=-\frac{a_i}{2(\operatorname{Im}w_i)^4}, \qquad 
        \Omega^{Ch}_{w_i\bar{w_i}z_i\bar{z_i}}=\frac{b_i}{4(\operatorname{Im}w_i)}.
    \end{align*}
    \item The Bismut Ricci form satisfies
    \begin{align*}
        \rho_B(\omega_h) &= -\frac{3}{4(\operatorname{Im}w_i)^2}dw_i\wedge d\bar{w}_i.
    \end{align*}
\end{enumerate}
\end{prop}

\subsection{Generalized K\"ahler structures on Oeljeklaus-Toma manifolds}

A \textit{generalized K\"ahler structure (GK structure)} \cite{GHR,GualtieriGKG} on a manifold $M$ is a triple $(g,I, J)$, where $g$ is a Riemannian metric compatible with two integrable complex structures $I$ and $J$, furthermore satisfying
    \begin{align*}
        d_I^c \omega_I &= -d_J^c\omega_J, \quad dd_I^c\omega_I = -dd_J^c\omega_J=0.
    \end{align*}
Hitchin \cite{HitchinPoisson} shows that a GK structure has an associated Poisson tensor
\begin{align*}
    \gs = \tfrac{1}{2} g^{-1} [I, J].
\end{align*}
We are interested here in the simplest case of GK manifolds, namely when $\gs$ vanishes.  This is equivalent to
    \begin{align*}
        [I, J]=0,
    \end{align*}
    and then we refer to the GK structure as \emph{commuting type}.  See \cite{ApostolovGualtieri} for further background on such structures.  We define
    \begin{align*}
        \Pi:= IJ \in \operatorname{End}(TM)
    \end{align*}
    where $\Pi^2=\operatorname{Id}$. Then, we can decompose $TM$ with respect to the eigenvalue $\pm 1$:
    \begin{align*}
        TM=T_+M\oplus T_-M.
    \end{align*}    
    Now, we introduce the GK-structure of commuting type on OT manifolds. On the universal cover, $\mathbb{H}^s\times\mathbb{C}^s$. Consider the invariant complex structures on $\mathbb{H}^s$ and $\mathbb{C}^t$, $J_{\mathbb{H}^s}$ and $J_{\mathbb{C}^s}$.  Let:
    \begin{align*}
        &I=\left (
        \begin{aligned}
            J_{\mathbb{H}^s} &  \\
            & J_{\mathbb{C}^s}
        \end{aligned}
        \right ),
        \qquad J=\left (
        \begin{aligned}
            J_{\mathbb{H}^s} &  \\
            & -J_{\mathbb{C}^s}
        \end{aligned}
        \right ).
    \end{align*}
    These are invariant complex structures on $\mathbb{H}^s\times\mathbb{C}^s$, hence pass to the quotient.  Moreover, $[I,J]=0$.  Defining $\Pi=IJ = -\operatorname{Id}_\mathbb{H}\oplus\operatorname{Id}_\mathbb{C}$, $E_\mathbb{C} = T\mathbb{C}^s$, $E_\mathbb{H} = T\mathbb{H}^s$. With the associated splitting as above, we see:
    \begin{align*}
        &TM = E_\mathbb{H}\oplus E_\mathbb{C}\\
        &T_-M = E_\mathbb{H}\\
        &T_+M = E_\mathbb{C}.
    \end{align*}
    In our case, since we have a specific split of tangent bundle, we will use the precise component notation ($\mathbb{C}$ or $\mathbb{H}$) instead of the standard notation $(+/-)$. Throughout this paper, we use coordinates $(w,z):=(w_1, \dots, w_s, z_1, \dots, z_s)$ on $\mathbb H^s \times \mathbb C^s$, with respect to the complex structure $I$.
    Note that we obtain differential operators 
    \begin{align*}
    d=&\ d_\mathbb{C}+d_\mathbb{H},\\
d_\mathbb{C}=&\ \pi_{\mathbb{C}} \circ d=\partial_z+\partial_{\bar{z}},\\
    d_\mathbb{H}=&\ \pi_{\mathbb{H}} \circ d =\partial_w+\partial_{\bar{w}},
    \end{align*}
    where $\pi_\mathbb{C}$ and $\pi_\mathbb{H}$ are the projections of $E_\mathbb{C}^*$ and $E_\mathbb{H}^*$, respectively.
    \begin{prop} The triple $(h^{a,b}, I, J)$ where $h^{a,b}$ is the metric determined by (\ref{e:AB model initial metric}), is generalized K\"ahler of commuting type.
    \end{prop}
    \begin{proof} We will give the proof for the standard metric $h$ in (\ref{e:model metric}), with the general case being analogous.
        First, note the pluriclosed metric $h$ is compatible with both complex structures $I$ and $J$, and $E_\mathbb{C}$ and $E_\mathbb{H}$ are $h$-orthogonal. Second, we let $\omega_I = h(\cdot,I\cdot)$, and $\omega_J = h(\cdot,J\cdot)$. Let $(w,z)$ be the complex coordinate with respect to $I$. Then:
        \begin{align*}
            &\omega_I = \omega_h = \sum_{i=1}^s\sqrt{-1}\frac{dw_i\wedge d\bar{w_i}}{(\operatorname{Im}w_i)^2} + \sqrt{-1}\operatorname{Im}w_i dz_i\wedge d\bar{z_i},\\
            &\omega_J = \sum_{i=1}^s\sqrt{-1}\frac{dw_i\wedge d\bar{w_i}}{(\operatorname{Im}w_i)^2} - \sqrt{-1}\operatorname{Im}w_i dz_i\wedge d\bar{z_i},
        \end{align*}
        and
        \begin{align*}
            &d_I^c \omega_I = \sqrt{-1}(\partial_{\bar{w}}+\partial_{\bar{z}} - \partial_w - \partial_{z}) \omega_I = -(\partial_{\bar{w}} - \partial_w)(\operatorname{Im}w_i dz_i\wedge d\bar{z_i}),\\
            &d_J^c\omega_J = \sqrt{-1}(\partial_{\bar{w}}+\partial_z - \partial_w - \partial_{\bar{z}})\omega_J = (\partial_{\bar{w}} - \partial_w)(\operatorname{Im}w_i dz_i\wedge d\bar{z_i}).
        \end{align*}
        Thus, $d_I^c\omega_I = -d_J^c \omega_J$. By direct computation, $dd_I^c\omega_I = -dd_J^c\omega_J=0$.  Hence, $(h,I,J)$ is a GK structure.
    \end{proof}

We observe that the metric satisfies the leafwise K\"ahler conditions $d_{\mathbb C} \left(\left. \gw_I \right|_{E_{\mathbb C}} \right) = 0$ and $d_{\mathbb H} \left(\left. \gw_I \right|_{E_{\mathbb H}} \right) = 0$, in line with the general theory of commuting-type GK structures developed by Apostolov-Gualtieri \cite{ApostolovGualtieri}.

\section{Pluriclosed flow on Oeljeklaus-Toma manifolds}
\subsection{Pluriclosed flow and its normalized flow}

The \textit{pluriclosed flow} is the evolution equation \cite{PCF}:
\begin{align}\label{e:PCF}
    \frac{\partial}{\partial t}\omega(t) = -\rho_B^{1,1}(\omega(t)),
\end{align}
where $\rho_B$ is the Bismut-Ricci form.  Motivated by the 
Type III behavior of the Ricci flow, we let the time parameter $s = e^t-1$ and consider the following blow-down family of metrics:
\begin{align*}
    \til{\omega}(t) = \frac{\omega(s)}{s+1},
\end{align*}
where $\omega(s)$ satisfies the pluriclosed flow (\ref{e:PCF}).  It follows that $\til{\omega}(t)$ satisfies
\begin{align}\label{e:PCF_normalized}
    \frac{\partial}{\partial t}\omega(t) = -\rho_B^{1,1}(\omega(t)) - \omega(t),
\end{align}
which we call the \textit{normalized pluriclosed flow}.

\subsection{Model solutions}\label{s:model cases}

The expected qualitative behavior of the normalized pluriclosed flow on OT manifolds is captured by solutions with initial data the model metrics $\gw_{h}^{a,b}$.  Straightforward computations show that the normalized pluriclosed flow with this initial data is
\begin{align}\label{e:normalized equation}
    \omega^{a,b}_h(t) = \sum_{i=1}^s \sqrt{-1}((1-e^{-t})\frac{3}{4}+e^{-t}a_i)\frac{1}{(\operatorname{Im}w_i)^2}dw_i\wedge d\bar{w}_i + \sqrt{-1}e^{-t}b_i\operatorname{Im}w_i dz_i\wedge d\bar{z}_i.
\end{align}
Later, we shall denote the time-dependent normalized model metric starting with $\omega^{a,b}_h$ as $\omega^{a,b}_h(t)$. We shall write $\omega_h(t)$ for short if $a_i$, $b_i=1$, for all $1\leq i\leq s$.  
Observe that for these model solutions the Chern torsion $T$ is uniformly bounded in time, i.e., $|T(t)|\leq C$. Moreover, it follows that
\begin{align*}
    \frac{\omega_h^{a,b}(t)}{t+1}\to \sum_{i=1}^s \frac{3}{4(\operatorname{Im}w_i)^2}dw_i\wedge d\bar{w}_i,
\end{align*}
which can be considered as a degenerate metric on $X(K, U)$.  As explained in \cite{fusi2024pluriclosed}, the blowdown manifolds converge in Gromov-Hausdorff sense to a torus $\mathbb T^s$ with a canonical flat metric $d(K, U)$ depending only on the algebraic field $K$ and the rank $s$ subgroup $U$.  We recall the result here, noting that OT manifolds are compact solvmanifolds and the model metrics are left-invariant:
\begin{thm} \label{t:homogOT} (\cite{fusi2024pluriclosed}) Let $\omega_0$ be a left-invariant pluriclosed metric on an OT manifold $M$, then the normalized pluriclosed flow starting with $\omega_0$ converges to $(\mathbb{T}^s,d(K, U))$ in the Gromov-Hausdorff sense.
\end{thm}

\noindent It is reasonable to conjecture that this behavior holds in the general case:
\begin{conj} \label{c:OT_PCF} Let $M = X(K,U)$ be an OT manifold of type $(s,s)$, then for any pluriclosed metric $\omega_0$, the normalized pluriclosed flow \ref{e:PCF_normalized} with initial metric $\omega_0$ exists on $[0, \infty)$, and converges to $(\mathbb{T}^s,d(K, U))$ in the Gromov-Hausdorff sense.
\end{conj}

We shall prove the first part of the conjecture in Section \ref{s:long time section}. For the second half of the conjecture, we have the following sufficiency condition.
\begin{prop}\label{p:GH convergence argument}
    Let $\omega(t)$ be the normalized pluriclosed flow solution on the OT manifold $X(K,U)$. Suppose that there exists a constant $C>0$ such that 
    \begin{itemize}
        \item $C^{-1}\leq \operatorname{tr}_{\omega(t)}\omega_h(t)\leq C$,
        \item $\lim\limits_{t\to\infty}\operatorname{tr}_{\omega(t)}\omega_h(t) = 1$ for all $x\in X(K,U)$.
    \end{itemize}
    Then, we have
    \begin{align*}
    (X(K,U),\omega(t))\to (\mathbb{T}^s,d(K, U)),
\end{align*}
to the flat metric $d(K,U)$ in the Gromov-Hausdorff sense.
\end{prop}
\begin{proof} We give a brief sketch as the result is already essentially contained in e.g. \cite{fusi2024pluriclosed}, \cite{ZhengOT}.  Note that for OT manifolds, there is a canonical fibration map:
    \begin{align*}
        F:X(K,U)\to \mathbb{T}^s
    \end{align*}
    where the fiber is diffeomorphic to $\mathbb{T}^{3s}$.  In particular, $E_\mathbb{C}$ will be in the kernel of $dF$. Since the quotient of $\{z\}\times \mathbb{C}^s$ is dense in the $\mathbb{T}^{3s}$ fiber (See Section 2 of \cite{VerbiOT}), the degenerate metric $\sum_{i=1}^s \frac{3}{4(\operatorname{Im}w_i)^2}dw_i\wedge d\bar{w}_i$ will induce a metric on the base $\mathbb{T}^s$. Let $d(K,U)$ be the metric on $\mathbb{T}^s$ induced from $\sum_{i=1}^s \frac{3}{4(\operatorname{Im}w_i)^2}dw_i\wedge d\bar{w}_i$, which is flat.

    Now we choose $G:\mathbb{T}^s\to X(K,U)$ to be any map so that $F\circ G = \operatorname{Id}$, and show that for $T$ large, $F$ and $G$ are $\epsilon(T)$-Gromov-Hausdorff approximations. By our assumption, the limit of $\omega(t)$ exists and is equal to $\sum_{i=1}^s \frac{3}{4(\operatorname{Im}w_i)^2}dw_i\wedge d\bar{w}_i$. Since the quotient of $\{z\}\times \mathbb{C}^s$ is dense in the $\mathbb{T}^{3s}$ fiber, the distance between two points in $F^{-1}(x)$ will converge to $0$ and the distance of any two representatives of two fibers will approximate to the distance $d(K,U)$, from which the Gromov-Hausdorff convergence follows.
\end{proof}
Thus, to prove the second half of the conjecture, it is enough to prove that $\operatorname{tr}_{\omega(t)} \omega_h(t)$ is uniformly bounded from above and below and converges to $1$.  In geometric flow theory, to obtain the long-time behavior, a Type III curvature estimate is usually needed. Then the collapsing behavior can be shown using the monotonicity formula and the Cheeger-Fukaya-Gromov collapsing theory (e.g., \cite{Lott1, LottDR}).  In our case, using the complex structure, $\operatorname{tr}_{\omega(t)}\omega_h(t)$ will satisfy a strictly parabolic equation (See Section \ref{s:long time section}), and thus it is possible to estimate it directly.  The Type III estimate for \eqref{e:PCF} starting with an arbitrary initial pluriclosed metric on the OT manifolds remains an interesting open question, even in the Inoue surface case ($s=1$).
See a recent study of the long-time behavior of the a related curvature flow following this route in \cite{GuoWang-alpha-Ricci-flow}.
\subsection{Long-time existence of pluriclosed flow}\label{s:long time section}

\begin{proof}[Proof of Theorem \ref{t:mainthm1}]
Let $M=X(K,U) = \mathbb{H}^s\times\mathbb{C}^s/U\ltimes \mathcal{O}_K$ be an OT manifold of type $(s,s)$. Since the universal cover $\til{M}\cong\mathbb{H}^s\times\mathbb{C}^s$ is a product manifold, and the action $U\ltimes \mathcal{O}_K$ acts on $\til{M}$ diagonally, the tangent bundle $TM = E_{\mathbb{H}}\oplus E_{\mathbb{C}}$, split as a direct sum of two subbundles assocaited with two projection maps, $\operatorname{pr}_{\mathbb{H}}:TM\to E_\mathbb{H}$ and $\operatorname{pr}_{\mathbb{C}}:TM\to E_\mathbb{C}$.  Now, given any Riemannian metric $g$ on $M$, denote $g_\mathbb{H} = g\circ \operatorname{pr}_\mathbb{H}$, the induced metric on $E_\mathbb{H}$. Similarly, $g_\mathbb{C} = g\circ \operatorname{pr}_\mathbb{C}$. Note that the model metric \ref{e:model metric} splits with respect to the splitting of $TM$, i.e., $h = h_\mathbb{H}\oplus h_\mathbb{C}$. Thus, the trace $\operatorname{tr}_h g = \operatorname{tr}_{h_\mathbb{C}\oplus h_\mathbb{H}} g= \operatorname{tr}_{h_\mathbb{C}} g_\mathbb{C} + \operatorname{tr}_{h_\mathbb{H}}g_\mathbb{H}$. We can decompose $\operatorname{tr}_g h$ in a similar way.  We let $\Upsilon(g, h)=\nabla_g-\nabla_h$, the difference of the Chern connection; $\Omega^g$ is the Chern curvature of metric $g$ and $\Delta$ is the Chern laplacian. Denote $\{w,z\}$ as the standard coordinate on $\mathbb{H}^s\times \mathbb{C}^s$. All other letters in the formulas represent general coordinates. Denote:
\begin{align*}
    Q_{i \bar{j}}=g^{\bar{l} k} g^{\bar{n} m} T_{i k \bar{n}} T_{\bar{j} l m}.
\end{align*}

\begin{lemma}\label{l:partial trace}
    Given the setup above,
    \begin{align}
        &\left(\frac{\partial}{\partial t} - \Delta \right) \operatorname{tr}_{g_\mathbb{H}}h_{\mathbb{H}} = -g^{r\bar{s}}g^{u\bar{q}}h_{w_i\bar{w_j}}\Upsilon_{ru}^{w_i}\Upsilon_{\bar{s}\bar{q}}^{\bar{w}_j}+g^{r\bar{s}}g^{w_i\bar{w}_j}\Omega^h_{r\bar{s}w_i\bar{w}_j}-g^{w_i\bar{q}}g^{p\bar{w}_j}h_{w_i\bar{w}_j}Q_{p\bar{q}},\label{eq:inverse trace h}\\
        & \left(\frac{\partial}{\partial t} - \Delta \right)  \operatorname{tr}_{h_\mathbb{H}}g_{\mathbb{H}} = -g^{r\bar{s}}h^{w_i\bar{w}_j}\Upsilon^{p}_{rw_i}\Upsilon^{\bar{q}}_{\bar{s}\bar{w}_j}g_{p\bar{q}}+h^{w_i\bar{w}_j}Q_{w_i\bar{w}_j}-g^{r\bar{s}}g_{w_i\bar{w}_j}h^{w_i\bar{q}}h^{p\bar{w}_j}\Omega^h_{r\bar{s}p\bar{q}},\label{eq:trace h}\\
        &\left(\frac{\partial}{\partial t} - \Delta \right)  \operatorname{tr}_{g_\mathbb{C}}h_{\mathbb{C}} = -g^{r\bar{s}}g^{u\bar{q}}h_{z_i\bar{z_j}}\Upsilon_{ru}^{z_i}\Upsilon_{\bar{s}\bar{q}}^{\bar{z}_j}+g^{r\bar{s}}g^{z_i\bar{z}_j}\Omega^h_{r\bar{s}z_i\bar{z}_j}-g^{z_i\bar{q}}g^{p\bar{z}_j}h_{z_i\bar{z}_j}Q_{p\bar{q}},\\
        &\left(\frac{\partial}{\partial t} - \Delta \right)  \operatorname{tr}_{h_\mathbb{C}}g_{\mathbb{C}} = -g^{r\bar{s}}h^{z_i\bar{z}_j}\Upsilon^{p}_{rz_i}\Upsilon^{\bar{q}}_{\bar{s}\bar{z}_j}g_{p\bar{q}}+h^{z_i\bar{z}_j}Q_{z_i\bar{z}_j}-g^{r\bar{s}}g_{z_i\bar{z}_j}h^{z_i\bar{q}}h^{p\bar{z}_j}\Omega^h_{r\bar{s}p\bar{q}}.
    \end{align}
\end{lemma}

\begin{proof}[Proof of Lemma \ref{l:partial trace}]
    The proof is almost identical to Lemma 4.3 of \cite{StreetsSTB}, except that the metric $g$ does not necessarily split with respect to $E_\mathbb{H}$. Here we sketch the computation of $\operatorname{tr}_{g_\mathbb{H}}h_{\mathbb{H}}$.  Note, $\rho_B^{1,1} = S-Q$, where $S = \operatorname{tr}_\omega \Omega^\omega$. Thus, using the local formula of the Chern curvature:
    \begin{align*}
        \frac{\partial}{\partial t}\operatorname{tr}_{g_\mathbb{H}}h_{\mathbb{H}} = g^{w_i\bar{q}}g^{p\bar{w}_j}h_{w_i\bar{w}_j}(-g_{p\bar{q},r\bar{s}}+g^{u\bar{v}}g_{p\bar{v},r}g_{u\bar{q},\bar{s}})g^{r\bar{s}}-g^{w_i\bar{q}}g^{p\bar{w}_j}h_{w_i\bar{w}_j}Q_{p\bar{q}}.
    \end{align*}
    Also, the Chern Laplacian acts on $\operatorname{tr}_{g_\mathbb{H}}h_{\mathbb{H}}$ by a standard formula:
    \begin{align*}
        \Delta \operatorname{tr}_{g_\mathbb{H}}h_{\mathbb{H}} = &g^{r\bar{s}}(g^{w_i\bar{v}}g^{u\bar{q}}g^{p\bar{w}_j}g_{u\bar{v},r}g_{p\bar{q},\bar{s}}h_{w_i\bar{w}_j}
        +g^{w_i\bar{q}}g^{p\bar{v}}g^{u\bar{w}_j}g_{u\bar{v},r}g_{p\bar{q},\bar{s}}h_{w_i\bar{w}_j}
        -g^{w_i\bar{q}}g^{p\bar{w}_j}g_{p\bar{q},r\bar{s}}h_{w_i\bar{w}_j}\\
        &-g^{w_i\bar{q}}g^{p\bar{w}_j}g_{p\bar{q},\bar{s}}h_{w_i\bar{w}_j,r}-g^{w_i\bar{q}}g^{p\bar{w}_j}g_{p\bar{q},r}h_{w_i\bar{w}_j,\bar{s}}
        +g^{w_i\bar{w}_j}h_{w_i\bar{w}_j,r\bar{s}}).
    \end{align*}
    Combining these two formulas, with the cancellation of the first order terms and the completing of squares, the result follows.  The other results are analogous.
\end{proof}
To show Theorem \ref{t:mainthm1}, we need to show that, for any time interval $[0,T]$, with $T>0$, the metrics $g(t)$ is $C(T)$-equivalent to $h$, i.e.,
\begin{align*}
    C(T)^{-1}h\leq g(t)\leq C(T) h
\end{align*}
for any $t\in [0,T]$. In other words, we need to show time dependent estimates of $\operatorname{tr}_gh$ and $\tr_{h} g$.  To achieve this, we need to use the Chern curvature of the model metric, moreprecisely the fact that the Chern curvature on $E_{\mathbb H}$ is negative.  In particular, using Proposition \ref{p:modelcurvature} and \eqref{eq:inverse trace h} we have
\begin{align*}
    \left(\frac{\partial}{\partial t} -\Delta \right)\operatorname{tr}_{g_\mathbb{H}}h_{\mathbb{H}} \leq - \frac{1}{2}(\operatorname{tr}_{g_\mathbb{H}}h_{\mathbb{H}})^2.
\end{align*}
By the maximum principle: $\operatorname{tr}_{g_\mathbb{H}}h_{\mathbb{H}}\leq \frac{1}{C_0+\frac{1}{2}t}$, where $C_0$ only depends on the initial metric $g(0)$.  Now, for the full trace, using Lemma \ref{l:partial trace} and the above estimate we have
\begin{align*}
    \left(\frac{\partial}{\partial t} - \Delta \right)\operatorname{tr}_g h &\leq \sum_{i=1}^s g^{z_i\bar{z}_i}g^{w_i\bar{w}_i}\Omega^h_{w_i\bar{w}_iz_i\bar{z}_i}\\
    &\leq \frac{1}{4}(\operatorname{tr}_{g_\mathbb{C}}h_\mathbb{C})(\operatorname{tr}_{g_\mathbb{H}}h_\mathbb{H})\\
    &\leq \frac{1}{4C_0+2t}\operatorname{tr}_g h.
\end{align*}
By the maximum principle again, 
\begin{align*}
    \operatorname{tr}_g h \leq C_0'(4C_0+2t)^{\frac{1}{2}},
\end{align*}
where $C_0'$ again only depends on $g(0)$. Thus, we have a time-dependent lower bound:
\begin{align*}
    g(t)\geq C_1(T) h.
\end{align*}
Now, for every fixed $T>0$, we can apply Theorem 1.8 of \cite{StreetsPCFBI}, we have a time dependent upper bound, $C_2(T)$. Hence:
\begin{align*}
    C_1(T) h\leq g(t)\leq C_2(T)h.
\end{align*}
Using the higher regularity of uniformly parabolic solutions of pluriclosed flow \cite{StreetsPCFBI, JordanStreets, garcia2023non}, the long-time existence follows.
\end{proof}

\section{Generalized K\"ahler-Ricci flow on Oeljeklaus-Toma manifolds}

In this section we consider the more restricted setting of GKRF on OT manifolds.  As explained in \S \ref{s:OT}, OT manifolds come equipped with generalized K\"ahler metrics of commuting type, and the pluriclosed flow starting with such a metric will preserve this structure \cite{GKRF}, leading to new estimates.

\subsection{Cohomology constraint and scalar reduction}

The key point established in \cite{StreetsSTB} is that after accounting for a certain cohomological constraint, the flow \ref{e:PCF_normalized} can be reduced to a scalar twisted Monge-Amp\`ere type equation.  We recall some aspects of this cohomology group (cf. \cite{StreetsPCFBI} Definition 7.3 and 7.4, \cite{GRFbook})
\begin{defn}[\cite{StreetsPCFBI}]\label{d:cohomology of splitting type}
    Let $(M^{2n},g,I,J)$ be a generalized Kähler manifold of commuting type.  Let
    \begin{align*}
        \mathcal{H}(M) = \{\omega\in \Lambda_{I,\mathbb{R}}^{1,1}|\sqrt{-1}\partial\bar{\partial}\omega = 0\}/\{\sqrt{-1}(\partial_+\bar{\partial}_+ - \partial_-\bar{\partial}_-)u|u\in C^\infty(M)\}.
    \end{align*}
Note $\omega\in [\omega_0]$, if and only if there exists a smooth function $f\in C^\infty(M)$ such that
\begin{align}\label{e:ddbar lemma for GK}
    \omega = \omega_0 + \sqrt{-1}(\partial_+ \bar{\partial}_{+} - \partial_-\bar{\partial}_-)f.
\end{align}
\end{defn}

Firstly, the dimension of the cohomology group $\mathcal{H}(M)$ (Definition \ref{d:cohomology of splitting type}) on the Inoue surface $S_M$ is known.
\begin{thm}[\cite{HaoJosh}] \label{t:haojosh}
    For a complex surface with split tangent bundle, $M$, then 
    \begin{align*}
        \mathcal H(S_M) = \left\{  [\omega_h^{a,b}]\ |\ a, b \in \mathbb R \right\}.
    \end{align*}
    In particular, $\dim \mathcal H(S_M) = 2$.
\end{thm}

For more general OT manifolds we prove a conjecturally sharp lower bound on the dimension of this space by exhibiting a space spanned by the model metrics.
\begin{prop}
    For positive constants, $a_i$, $b_i$, and $a_i'$, $b_i'$, such that there exists $i$ with either $a_i \neq a_i'$ or $b_i\neq b_i'$, then
    \begin{align*}
        [\omega^{a,b}_h]\neq [\omega^{a',b'}_h].
    \end{align*}
    \begin{proof}
    It suffices to show that for constants $a_i$ and $b_i$ such that $a_i$, $b_i\in \mathbb{R}$, if there exists $u\in C^\infty(X(K,U))$ such that on the universal cover $\mathbb{H}^s\times\mathbb{C}^s$,
    \begin{align*}
        ({\partial_z\partial_{\bar{z}}}-{\partial_w\partial_{\bar{w}}}) u = \sum_{i=1}^sa_i\frac{dw_i\wedge d\bar{w_i}}{(\operatorname{Im}w_i)^2} + b_i\operatorname{Im}w_i dz_i\wedge d\bar{z_i},
    \end{align*}
    then $a_i = b_i = 0$.  Tracing the above equation with $\omega^{a,b}_J$ yields
    \begin{align*}
        \Delta^{Ch}_{\omega_I^{a,b}} u = \mbox{const}
    \end{align*}
    on $X(K,U)$. By the maximum principle, $u\equiv C$, hence $a_i = b_0 = 0$ as required.
    \end{proof}
\end{prop}


The relevance of the above cohomology group stems from its relationship to the Bismut Ricci form for commuting-type generalized K\"ahler metrics.  In particular from Proposition 8.20 of \cite{GRFbook}, and specializing to the case of OT manifolds here, for such metrics $\rho_B^{1,1}$ can be decomposed: 
\begin{align}\label{e:bismut ricci 1,1 decomposition}
    \rho_B^{1,1}(\omega) = \rho_\mathbb{C}(\omega_\mathbb{C}) - \rho_\mathbb{H}(\omega_\mathbb{C})-\rho_\mathbb{C}(\omega_\mathbb{H})+\rho_\mathbb{H}(\omega_\mathbb{H}).
\end{align}
Moreover this leads to the following transgression formula for commuting-type GK metrics:
\begin{align}\label{e:Bismut difference}
    \rho_B^{1,1}(\omega_1)-\rho_B^{1,1}(\omega_2) = -\sqrt{-1}(\partial_z\partial_{\bar{z}} - \partial_w\partial_{\bar{w}})\log\frac{\det g^1_\mathbb{C}\cdot\det g^2_\mathbb{H}}{\det g^1_\mathbb{H}\cdot\det g^2_\mathbb{C}}.
\end{align}

\begin{defn}[Model flow] To simplify notation we set $P = \rho_B^{1,1}(h)$, where $h$ is the model metric (\ref{e:model metric}).  Let $\til{\omega}(t)=e^{-t}\til{\omega}_0 -(1-e^{-t})P$, where $\til{\omega}_0$ is a representative in $[\omega(0)]$, and $P:=\rho_B^{1,1}(\omega_h)=\sum_{i=1}^s-\frac{3\sqrt{-1}}{4}h_{w_i\bar{w}_j}dw_i\wedge d\bar{w}_j$, the Bismut curvature of the model metric (\ref{e:model metric}).
\end{defn}
Analogous to the K\"ahler Ricci flow, we can reduce the GKRF to a parabolic fully nonlinear twisted Monge-Amp\'ere equation.
\begin{lemma} \label{l:scalar_reduction}
    Suppose $\omega(t)=\til{\omega}(t)+\sqrt{-1}(\partial_z\partial_{\bar{z}}-\partial_w\partial_{\bar{w}})\phi$, where $\phi$ satisfies
    \begin{align}\label{e:potential eq}
      \frac{\partial}{\partial t}\phi(t)+\phi(t)=\operatorname{log}\frac{\operatorname{det}g_\mathbb{C}}{\operatorname{det}h_\mathbb{C}}-\operatorname{log}\frac{\operatorname{det}g_\mathbb{H}}{\operatorname{det}h_\mathbb{H}}  + c(t),
    \end{align}
    where $c(t)$ is some time-dependent constant. Then, $\omega_t$ solves \ref{e:PCF_normalized}.
\end{lemma}
\begin{proof}

    Use the transgression formula \eqref{e:Bismut difference}, we have:
    \begin{align*}
    \frac{\partial}{\partial t}\omega(t)&=-e^{-t}\omega_0-e^{-t}P+P-P+\sqrt{-1}(\partial_z\partial_{\bar{z}}-\partial_w\partial_{\bar{w}})\frac{\partial}{\partial t}\phi(t)\\
    &=-\omega(t)+\sqrt{-1}(\partial_z\partial_{\bar{z}}-\partial_w\partial_{\bar{w}})(\frac{\partial}{\partial t}\phi+\phi)-\rho_B^{1,1}(\omega_h)\\
    &=-\omega(t)-\rho_B^{1,1}(\omega(t)).
  \end{align*}
\end{proof}

\subsection{A priori estimates}

Using the scalar reduction of Lemma \ref{l:scalar_reduction} we prove new a priori estimates.  First, by a specific choice of $c(t)$ we obtain a $C^0$ estimate for the potential.
\begin{lemma}\label{l:C0 bound}
    If we choose $c(t)=st+ s\operatorname{log}\frac{3}{4}-\operatorname{log}c_1$, where $c_1\geq \sup\frac{\operatorname{det}g_{\mathbb{C}}^0}{\operatorname{det}h_\mathbb{C}}$, then:
    \begin{align*}
        -C_0\leq\phi(t)\leq C_0 e^{-t}(1+Ct),
    \end{align*}
    for some $C_0$, $C>0$.
\end{lemma}
\begin{proof}
    Recall: $\omega(t)=\til{\omega}(t)+\sqrt{-1}(\partial_z\partial_{\bar{z}}-\partial_w\partial_{\bar{w}})\phi$. From our choice, we have $\omega(0) = \omega_0$ and $\phi(0) = 0$. At the maximum point, with the choice of $c(t)$, we have:
    \begin{align*}
        \frac{\partial}{\partial t}(\max\phi(t)) &\leq \operatorname{log}\frac{\operatorname{det(\til{g}_\mathbb{C})/c_1}}{\operatorname{det} e^{-t} h_\mathbb{C}} - \operatorname{log}\frac{\operatorname{det}\til{g}_\mathbb{H}}{\operatorname{det} \frac{3}{4} h_\mathbb{H}} - \max \phi(t)\\
        &\leq \operatorname{log}\frac{\operatorname{det}e^{-t}g_\mathbb{C}^0/c_1}{\operatorname{det} e^{-t}h_\mathbb{C}} - \operatorname{log}\frac{\operatorname{det}(e^{-t}g_\mathbb{H}^0 + \frac{3}{4}(1-e^{-t})h_\mathbb{H})}{\operatorname{det}\frac{3}{4}h_\mathbb{H}} - \max\phi(t).
    \end{align*}
    To estimate the first two terms, observe
    \begin{align*}
e^t\operatorname{log}\frac{\operatorname{det}\frac{3}{4}h_\mathbb{C}}{\operatorname{det}((C_1e^{-t}+\frac{3}{4}(1-e^{-t}))h_\mathbb{C}}\leq e^ts\operatorname{log} \frac{1}{1+C_1e^{-t}}\leq C_2.
    \end{align*}
    By the maximum principle, since we choose $c_1\geq \sup\frac{\operatorname{det}g_{\mathbb{C}}^0}{\operatorname{det}h_\mathbb{C}}$, then $\phi(x,t)\leq f(t)$, where $f$ satisfies:
    \begin{align*}
        \frac{\partial}{\partial t}(e^tf(t)) = C_2.
    \end{align*}
    Thus, we have $\phi(t)\leq C_0e^{-t}(1+Ct)$, for some $C_0$, $C>0$.

    At a minimum point, note that $\frac{\operatorname{det}g_{\mathbb{C}}^0}{c_1\operatorname{det}h_\mathbb{C}}\geq c_0$, for some $c_0 >0$, we have:
    \begin{align*}
        \frac{\partial}{\partial t}(e^t\min\phi(t))\geq e^ts(\operatorname{log}\frac{C_3}{1+C_1 e^{-t}})\geq-C_4 e^t.
    \end{align*}
    Thus, the lower bound is obtained similarly.
\end{proof}

\begin{rmk}\label{sign} We note that for the  Chern Ricci flow on Inoue surfaces \cite{FTWZ} it is possible to get decaying upper and lower bounds directly via the maximum principle.  It is unclear how to directly achieve this for the twisted parabolic Monge-Amp\`ere equation here, a manifestation of the mixed convex/concave structure.
\end{rmk}

If we assume that the initial metric is in $[\omega^{a,b}_h]$ (ref: \eqref{e:AB model initial metric}), we have a refined $C^0$ estimate.
\begin{cor}\label{c:imporved C0}
    If there exist $a^i, b^i$ such that $\omega_0\in[\omega^{a,b}_h]$, then we can choose $c(t) = st+s\operatorname{log}\frac{3}{4}$ so that:
    \begin{align*}
        -C_0e^{-t}(1+Ct)\leq\phi(t)\leq C_0 e^{-t}(1+Ct),
    \end{align*}
    for some $C_0$, $C>0$.
\end{cor}
\begin{proof}
    We can choose the initial representative $\omega_0 = \omega^{a,b}_h$. In this case, $\phi(0)$ is a function satisfies $\omega(0)=\omega_0 +\sqrt{-1}(\partial_z\partial_{\bar{z}} - \partial_w \partial_{\bar{w}})\phi(0)$. Also, note that $\rho^{1,1}(\omega^{a,b}_h) = P$. Thus, we can choose the representative cohomological background metric $h$ to be $h^{a,b}$ as well. The rest of the proof is the same as Lemma \ref{l:C0 bound}. Note that in this case, we do not need to choose the normalization constant $c_1$ since the initial background metric is exactly the model metric $\omega^{a,b}_h$.
\end{proof}
\begin{rmk}
    The choice of $c(t)$ can be seen in a more geometric way. If we denote $h(t)$ to be the normalized model flow of the model flow in Section \ref{s:model cases}, then we see that $P = \rho^{1,1}(\omega_h(t))$, which is a constant. With the choice of $c(t) = st+s\log{\frac{3}{4}}$. Thus, we can express:
    \begin{align*}
        \frac{\partial}{\partial t}\phi(t)+\phi(t)=\operatorname{log}\frac{\operatorname{det}g_\mathbb{C}(t)}{\operatorname{det}h_\mathbb{C}(t)}-\operatorname{log}\frac{\operatorname{det}g_\mathbb{H}(t)}{\operatorname{det}h_\mathbb{H}(t)},
    \end{align*}
    which is measured with respect to the moving model metric $h(t)$. And the choice of $\log c_1$ is a normalizing constant to match the initial condition. Since the Gromov-Hausdorff limit, at least from all the existing results, is independent of the initial metric. Then the potential with respect to the moving model metric should converge to 0 uniformly. Thus, the choice of $c(t)$ is not surprising, and if $\omega_0$ is in the cohomology class of $\omega_h$, then we will achieve better decaying $C^0$ estimate.
\end{rmk}
Also, we need the estimate for $\dot{\phi}$, the time derivative of the potential. To derive the equation, several lemmas are needed.
\begin{lemma}\label{l:split potential eq}
    For the metrics $g_\mathbb{H}$ and $g_\mathbb{C}$ on bundles $E_\mathbb{H}$ and $E_\mathbb{C}$ respectively, we have the following evolution equation:
    \begin{align}
        &(\partial_t -\Delta) \operatorname{log}\frac{\operatorname{det}g_{\mathbb{H}}}{\operatorname{det}h_{\mathbb{H}}} = \frac{1}{2}|T|^2 + \frac{1}{2}\operatorname{tr}_{g_\mathbb{H}}h_{\mathbb{H}} - s,\label{e:potential h}\\
        &(\partial_t -\Delta) \operatorname{log}\frac{\operatorname{det}g_{\mathbb{C}}}{\operatorname{det}h_{\mathbb{C}}} = \frac{1}{2}|T|^2 - \frac{1}{4}\operatorname{tr}_{g_\mathbb{H}}h_{\mathbb{H}} - s\label{e:potential c}.
    \end{align}
\end{lemma}
\begin{proof}
    By the computation of Proposition \ref{p:modelcurvature}, for metrics on ($E_\mathbb{H}$ and $E_\mathbb{C}$), we have:
    \begin{align*}
        &\rho(h_\mathbb{H}) = \sum_{i=1}^s\frac{1}{4}\frac{1}{(\operatorname{Im}w_i)^2}dw_i\wedge d\bar{w}_i\\
        &\rho(h_\mathbb{C}) = \sum_{i=1}^s-\frac{1}{2}\frac{1}{(\operatorname{Im}w_i)^2}dw_i\wedge d\bar{w}_i
    \end{align*}
    Thus, by computing the trace, we have $\operatorname{tr}_g \rho_\mathbb{H} = \frac{1}{4}\operatorname{tr}_{g_\mathbb{H}}h_\mathbb{H}$ and $\operatorname{tr}_g \rho_\mathbb{C} = -\frac{1}{2}\operatorname{tr}_{g_\mathbb{H}}h_\mathbb{H}$. By using Lemma 4.1 from \cite{StreetsSTB}, it is done.
\end{proof}

A direct corollary of the above evolution equations is the a priori lower bound for the metric on $E_\mathbb{H}$:
\begin{cor}\label{c:lower z}
    There exists a constant $C$, depends on the initial metric, $g_0$, such that, on $E_\mathbb{H}$:
    \begin{align*}
        g_\mathbb{H}\geq C h_\mathbb{H}.
    \end{align*}
\end{cor}
\begin{proof}
    On $E_\mathbb{H}$, by the AM-GM inequality:
    \begin{align*}
        \operatorname{tr}_{g_\mathbb{H}}h_\mathbb{H}/s \geq (\frac{\operatorname{det}h_\mathbb{H}}{\operatorname{det}g_\mathbb{H}})^{\frac{1}{s}}.
    \end{align*}
    Thus, for \eqref{e:potential h}, we have:
    \begin{align*}
        (\partial_t -\Delta) \operatorname{log}\frac{\operatorname{det}g_{\mathbb{H}}}{\operatorname{det}h_{\mathbb{H}}} \geq \frac{1}{2}(\frac{\operatorname{det}h_\mathbb{H}}{\operatorname{det}g_\mathbb{H}})^{\frac{1}{s}}s -s.
    \end{align*}
    By the ODE comparison, we have:
    \begin{align*}
        \operatorname{tr}_{h_\mathbb{H}}g_\mathbb{H}\geq s (\frac{\operatorname{det}g_\mathbb{H}}{\operatorname{det}h_\mathbb{H}})^{\frac{1}{s}}\geq (\frac{1}{2} + C_0 e^{-t})^s\geq C.
    \end{align*}
\end{proof}
Now, we look at equation of $\dot{\phi}$:
\begin{lemma}
    For the time derivative of the potential, $\dot{\phi}$, we have:
    \begin{align*}
        (\frac{\partial}{\partial t} - \Delta)(\dot{\phi} + \phi) = 
        -\frac{3}{4}\operatorname{tr}_{g_\mathbb{H}}h_\mathbb{H}+s.
    \end{align*}
\end{lemma}
\begin{proof}
    Using \eqref{e:potential eq} and Lemma \ref{l:split potential eq}, we have:
    \begin{align*}
        (\frac{\partial}{\partial t} - \Delta)(\dot{\phi} + \phi) = -\frac{1}{4}\operatorname{tr}_{g_\mathbb{H}}h_\mathbb{H} - \frac{1}{2}\operatorname{tr}_{g_\mathbb{H}}h_\mathbb{H} + \frac{\partial}{\partial t}c(t) = -\frac{3}{4}\operatorname{tr}_{g_\mathbb{H}}h_\mathbb{H}+s.
    \end{align*}
\end{proof}
Now, we derive the upper bound estimate for $\dot{\phi}$:
\begin{lemma}\label{l:potential derivative upper bound}
    There exists a constant $C>0$, such that:
    \begin{align*}
        \dot{\phi} \leq C. 
    \end{align*}
\end{lemma}
\begin{proof}
    Note that, from \eqref{e:ddbar lemma for GK}, the Chern Laplacian on potential $\phi$ can be purely expressed as the geometric quantity:
    \begin{align*}
        \Delta\phi &= \operatorname{tr}_{g_\mathbb{C}}\partial_z\partial_{\bar{z}}\phi + \operatorname{tr}_{g_\mathbb{H}}\partial_w\partial_{\bar{w}}\phi = \operatorname{tr}_{g_\mathbb{C}}(g_\mathbb{C} - \til{g}_\mathbb{C}) + \operatorname{tr}_{g_\mathbb{H}}(\til{g}_\mathbb{H}-g_\mathbb{H})\\
        & = -e^{-t}\operatorname{tr}_{g_\mathbb{C}}g^0_\mathbb{C} + e^{-t}\operatorname{tr}_{g_\mathbb{H}}g^0_\mathbb{H} + (1-e^{-t})\frac{3}{4}\operatorname{tr}_{g_\mathbb{H}}h_\mathbb{H}.
    \end{align*}

    Consider the quantity $\dot{\phi} - (C_1-1)\phi = \dot{\phi}+\phi - C_1\phi$, for some constant $C_1>1$, which will be chosen later. Then:
    \begin{align*}
        (\frac{\partial}{\partial t} - \Delta)(\dot{\phi}- (C_1 -1)\phi) = &-\frac{3}{4}\operatorname{tr}_{g_\mathbb{H}}h_\mathbb{H} + s - C_1\phi\\
        &-C_1e^{-t}\operatorname{tr}_{g_\mathbb{C}}g^0_\mathbb{C} + C_1e^{-t}\operatorname{tr}_{g_\mathbb{H}}g^0_\mathbb{H} + C_1(1-e^{-t})\frac{3}{4}\operatorname{tr}_{g_\mathbb{H}}h_\mathbb{H}\\
        =&s-C_1\dot{\phi}-C_1e^{-t}\operatorname{tr}_{g_\mathbb{C}}g^0_\mathbb{C} + C_1e^{-t}\operatorname{tr}_{g_\mathbb{H}}g^0_\mathbb{H} + (C_1(1-e^{-t})-1)\frac{3}{4}\operatorname{tr}_{g_\mathbb{H}}h_\mathbb{H}.
    \end{align*}
    If for $A>0$, a very large constant, such that $\dot{\phi} - (C_1-1)\phi>A$. Then, by Lemma \ref{l:C0 bound}, 
    \begin{align*}
        \dot{\phi}\geq (C_1-1)C_0+A = A'
    \end{align*}
    for some large constant $A'>0$. By the lower bound of the metric on $E_\mathbb{H}$ (Corollary \ref{c:lower z}), we have:
    \begin{align*}
         (\frac{\partial}{\partial t} - \Delta)(\dot{\phi}- (C_1 -1)\phi)\leq s+C_1C_0 - C_1\dot{\phi},
    \end{align*}
    for some constant $C_0 = C_0(s,g(0))>0$. Thus, for $A$ large, by the maximum principle:
    \begin{align*}
        \dot{\phi} - (C_1 - 1)\phi \leq A.
    \end{align*}
    By Lemma \ref{l:C0 bound} again, done.
\end{proof}

\subsection{Results on the Inoue surface}

When $s = 1$, then the OT manifolds $X(K,U)$ become the well-known Inoue surfaces of type $S_M$. On these surfaces, several estimates can be strengthened.  First note that by Theorem \ref{t:haojosh}, for any GK initial metric $\omega_0$, there will be constants $a$ and $b$ such that $\omega_0\in[\omega^{a,b}_h]$.
Thus, we always have the improved $C^0$ estimate for $\phi$ from Corollary \ref{c:imporved C0}.
\begin{cor}\label{c:Inoue C0}
    Let $\omega_0$ be any GK metrics on the Inoue surface $S_M$. Then, if we choose $c(t) = t + \log(\frac{3}{4})$, then:
    \begin{align*}
        -C_0e^{-t}(1+Ct)\leq\phi(t)\leq C_0 e^{-t}(1+Ct),
    \end{align*}
    for some $C_0$, $C>0$.
\end{cor}

Secondly, when $s = 1$, then $E_\mathbb{C}$ and $E_\mathbb{H}$ are holomorphic line bundles. Thus, the AM-GM inequality becomes equality in this case, and we can have the following lower bound of $\dot{\phi}$.
\begin{lemma}\label{l:potential derivative lower bound}
    On Inoue surface $S_M$, for the potential $\phi$, there exists a constant $C>0$, such that:
    \begin{align*}
        \dot{\phi}\geq -C.
    \end{align*}
\end{lemma}
\begin{proof}
    Note $s=1$. Choose a large constant $\Lambda>0$, consider the quantity $\dot{\phi} + \phi+\frac{1}{\Lambda}\phi$. Then
    \begin{align*}
        (\frac{\partial}{\partial t} - \Delta)(\dot{\phi}+\phi+\frac{1}{\Lambda}\phi) = 1 +\frac{1}{\Lambda}\dot{\phi}+ \frac{1}{\Lambda}e^{-t}\operatorname{tr}_{g_\mathbb{C}}g^0_\mathbb{C} - \frac{1}{\Lambda}e^{-t}\operatorname{tr}_{g_\mathbb{H}}g^0_\mathbb{H} - (\frac{1}{\Lambda}(1-e^{-t})+1)\frac{3}{4}\operatorname{tr}_{g_\mathbb{H}}h_\mathbb{H}.
    \end{align*}
    Now, for $A>0$, a very large constant, such that $\dot{\phi}+\phi +\frac{1}{\Lambda}\phi < -A$. By \eqref{e:potential eq}, notice that in this case $s=1$, we know:
    \begin{align*}
        \dot{\phi} +\phi = \log\frac{\det g_\mathbb{C}/c_1}{\det e^{-t}h_\mathbb{C}} - \log\frac{\det g_\mathbb{H}}{\det \frac{3}{4}h_\mathbb{H}} = -\log(c_1e^{-t}\operatorname{tr}_{g_\mathbb{C}}h_\mathbb{C}) + \log (\frac{3}{4}\operatorname{tr}_{g_\mathbb{H}}h_\mathbb{H}).
    \end{align*}
    By Corollary \ref{c:Inoue C0}, in particular $|\phi|\leq C_0$. Then, there is a large constant $A' = A- C_0>0$ such that $\dot{\phi}\leq -A'$, and 
    \begin{align*}
        \operatorname{tr}_{g_\mathbb{H}}h_\mathbb{H}\leq \frac{4}{3}e^{\frac{C_0}{\Lambda}-A}e^{-t}c_1 \operatorname{tr}_{g_\mathbb{C}}h_\mathbb{C}\leq \frac{4}{3}e^{\frac{C_0}{\Lambda}-A}c_0e^{-t}\operatorname{tr}_{g_\mathbb{C}}g^0_\mathbb{C},
    \end{align*}
    for some constant $c_0$, only depends on the initial metric.
    By choosing $\Lambda$ large, such that $1-\frac{C_0}{\Lambda}>\frac{2}{3}$, fixed, when $A$ is large enough, $c_0e^{\frac{C_0}{\Lambda}-A} \leq \frac{1}{\Lambda}$. Thus, choose an $A = \frac{\Lambda}{2}$ large enough, such that $\dot{\phi}(t)\geq -\frac{1}{2}A$, for $t\in[0,T]$, where $T$ is large enough. Then for the first time $\dot{\phi}+\phi +\frac{1}{\Lambda}\phi = -A$, we have:
    \begin{align*}
        (\frac{\partial}{\partial t}-\Delta)(\dot{\phi} + \phi +\frac{1}{\Lambda}\phi)&\geq (\frac{1}{\Lambda } - \frac{C_1e^{-t}}{\Lambda}-c_0e^{\frac{C_0}{\Lambda}-A})e^{-t}\operatorname{tr}_{g_\mathbb{C}}h_\mathbb{C} + 1+ \frac{1}{\Lambda}\dot{\phi}\\
        &>0.
    \end{align*}
    By the maximum principle, the result follows.
\end{proof}
Combine with the potential derivative upper bound (Lemma \ref{l:potential derivative upper bound}), we have the following metric lower bound.
\begin{thm}\label{t:Inoue metric lower bound}
    On an Inoue surface $S_M$, for generalized K\"ahler metric $\omega_0$, there is a constant $C$, such that the normalized pluriclosed flow solution $\omega(t)$with initial data $\gw_0$ will satisfy
    \begin{align*}
        \omega(t)\geq C\omega_h(t),
    \end{align*}
    where $\omega_h(t)$ is the model flow.
\end{thm}
\begin{proof}
    By Lemma \ref{l:potential derivative upper bound} and Lemma \ref{l:potential derivative lower bound}, we have that:
    \begin{align*}
        -C\leq\dot{\phi} +\phi \leq C
    \end{align*}
    for some constant $C$. Now, from the equation \eqref{e:potential eq} and $s=1$, we have:
    \begin{align*}
        -C\leq \operatorname{tr}_{g_\mathbb{H}}h_\mathbb{H}/\operatorname{tr}_{g_\mathbb{C}}e^{-t}h_\mathbb{C}\leq C.
    \end{align*}
     Note that in the normalized pluriclosed flow model case, \eqref{e:normalized equation}, the $E_\mathbb{C}$ part will shrink at the rate of $e^{-t}$, while the $E_\mathbb{H}$ part will be equivalent to $h_\mathbb{H}$. Thus, by Corollary \ref{c:lower z}, we have the desired lower bound for $E_\mathbb{C}$.
\end{proof}
\begin{proof} [Proof of Theorem \ref{t:mainthm2}]
Combine Lemma \ref{l:C0 bound}, Corollary \ref{c:imporved C0}, Lemma \ref{l:potential derivative upper bound}, Corollary \ref{c:Inoue C0}, Lemma \ref{l:potential derivative lower bound}, and Theorem \ref{t:Inoue metric lower bound}, the results follow.
\end{proof}

\section{Closing remarks}

As pluriclosed flow is in particular a solution to generalized Ricci flow \cite{PCFReg}, there are scalar curvature monotonicity formulas as explained in \cite{Streetsscalar}.  The key extra input is a solution of the \emph{dilaton flow}
\begin{align*}
    \dt \psi = \gD \psi + \tfrac{1}{6} \brs{H}^2.
\end{align*}
To state the  monotonicity, we recall the notation of weighted Ricci and scalar curvatures:
\begin{align*}
&\operatorname{Ric}^{H,\psi}=\operatorname{Ric}-\frac{1}{4}H^2 + \operatorname{Hess}\psi- \frac{1}{2}(d_g^* H + i_{\nabla \psi}H)\\
    &R^{H,\psi}=R-\frac{1}{12}|H|^2 + 2\Delta \psi -|\nabla \psi|^2
\end{align*}
Along the \emph{normalized} generalized Ricci flow we then obtain 
\begin{cor}(cf. \cite{Streetsscalar} Prop 1.1) If $(g_t, H_t)$ is a solution of normalized generalized Ricci flow and $\psi$ is a solution of the dilaton flow, we have
\begin{align*}
    \left(\dt - \Delta \right) R^{H,\psi} = 2 |\operatorname{Ric}^{H,\psi}|^2 + R^{H,\psi},
\end{align*}
and the estimate:
\begin{align*}
    R^{H,\psi} \geq \inf_{M\times\{0\}}R^{H,\psi}e^{t}.
\end{align*}
\end{cor}

For a solution to pluriclosed flow there is a natural class of solutions to the dilaton flow, namely the induced metrics on flat line bundles (cf. 
\cite{StreetsPCFBI} Lemma 6.1).  On Inoue surfaces for instance there is a natural such bundle $\det E_{\mathbb H} \otimes (\det E_{\mathbb C})^{\otimes 2}$.  In particular by direct computations it can be shown that
\begin{align*}
    \psi = \frac{1}{9} \left(\operatorname{log}\frac{\det g_\mathbb{H}}{\det h_\mathbb{H}}+2\log\frac{\det g_\mathbb{C}}{\det e^{-t}h_\mathbb{C}} \right)
\end{align*}
is a solution of the dilaton flow after appropriate normalization and gauge-fixing.   In particular,  nonnegativity of $R^{H,\psi}$ is preserved.  Noting that in context now all the data defining $R^{H,\psi}$ is canonically determined by a pluriclosed metric alone, it will be interesting to ask whether $X(K,U)$ admits a metric such that $R^{H,\psi} \geq 0$, which becomes an elliptic problem. 
\begin{conj}
    On an OT manifold $X(K,U)$ there exists no pluriclosed metric with $R^{H,\psi}\geq0$.
\end{conj}
\noindent This question bears a resemblence to a result of Albanese \cite{AlbYambeInoue} showing that the Inoue surface $S_M$ does not have any positive scalar curvature metric.


\end{document}